\documentclass[12pt,leqno]{amsart}
\usepackage{amsmath,amsfonts,amssymb,amsthm}
\usepackage{graphicx}
\graphicspath{ }
\usepackage{wrapfig}
\usepackage{accents}
\usepackage{caption}
\usepackage{subcaption}
\usepackage{calligra}

\numberwithin{figure}{section}
\newcommand\norm[1]{\left\lVert#1\right\rVert}

\theoremstyle{plain}
\newtheorem{thm}{Theorem}[section]
\newtheorem{lem}[thm]{Lemma}

\newtheorem{cor}{Corollary}[thm]
\theoremstyle{definition}
\newtheorem{defn}{Definition}[section]

\newtheorem{exmp}{Example}[section]

\theoremstyle{remark}

\usepackage{mathtools}
\title{Non-existence of certain type of convex functions on a Riemannian manifold with a pole}
\author[A. A. Shaikh, C. K. Mondal and I. Ahmad]{Absos Ali Shaikh$^1$$^*$, Chandan Kumar Mondal$^2$ and Izhar Ahmad$^3$}
\address{\noindent\newline $^{1,2}$Department of Mathematics,\newline University of
Burdwan, Golapbag,\newline Burdwan-713104,\newline West Bengal, India}
\email{aask2003@yahoo.co.in, aashaikh@math.buruniv.ac.in}
\email{chan.alge@gmail.com}

\address{\noindent\newline $^3$Department of Mathematics and Statistics,\newline King Fahd University of
Petroleum and Minerals,\newline Dhahran-31261,\newline Saudi Arabia}
\email{drizhar@kfupm.edu.sa}
\begin{document}
\begin{abstract}
This paper is devoted to the study of non-existence of certain type of convex functions on a Riemannian manifold with a pole. To this end, we have developed the notion of odd and even function on a Riemannian manifold with a pole and proved the non-existence of non-trivial and non-negative differentiable odd convex function whose gradient is complete. Finally, we have deduced some isoperimetric type inequality related with convex function.
\end{abstract}
\noindent\footnotetext{ $^*$ Corresponding author.\\
$\mathbf{2010}$\hspace{5pt}Mathematics\; Subject\; Classification: 52A20, 52A30, 52A42, 53B20, 53C20, 53C22. }
\maketitle
\section{Introduction}
The notion of convexity plays a crucial role in the field of economics, management science and optimization problems etc.  For instance, a (strictly) convex function on an open set of the Euclidean space has not more than one minimum. Consequently, the study of convex function and other related concepts of convexity are essential for both the theoretical and practical point of view. To model different real world problems, various researchers have generalized the concept of convexity assumption and many articles have appeared in the literature involving generalization of convex functions. 
\par Throughout this paper by $M$ we mean a complete Riemannian manifold of dimension $n$ endowed with some positive definite metric $g$ unless otherwise stated. To generalize various concepts of Euclidean space to M, straight line segment is substituted by a geodesic and vector space by $M$. There are many convex functions with several structural implication on $M$ and such functions form an important bridge between analysis and geometry. Since convex functions on a compact manifold reduce to constant, the investigation of such functions is more interesting on a non-compact manifold. A full discussion of convex functions on $M$ can be found in Rapcsak \cite{RAP97} and Udri\c{s}te \cite{UDR94}. After that many authors generalized the notion of convexity, see \cite{IAA12, AAS11, AW15, CA17}.

\par In the literature of differential geometry the existence of proper convex functions on a $M$ is a long standing problem. The first solution of this problem appears from the work of Bishop and O'Neill \cite{BN69}. They proved that if $M$ has finite volume, then it does not possess any non-trivial smooth convex function. Later, Yau \cite{YAU74} modified their result and proved the non-existence of continuous convex function on such $M$. In 2017, Neto et. al. \cite{NMS17} showed that if the geodesic flow on $M$ with finite volume is conservative then convex functions on such $M$ are constants. In the case of infinite volume, Greene and Wu \cite{GW76} presented that there always exists a $C^\infty$ Lipschitz continuous strictly convex exhaustion function on a complete non-compact $M$ whose sectional curvature is positive everywhere. Borsuk-Ulam Theorem \cite{MAT03} implies that real valued continuous function on an unit $n$-dimensional sphere attains the same value at two antipodal points. By using Borsuk-Ulam Thorem and negative gradient flow of convex function on $M$, we prove the non-existence of odd convex functions on $M$ with a pole. 
\par The paper is structured as follows. Section 2 deals with some well known facts about $M$. Section 3 is concerned with the convex functions on $M$ with pole. In this section we define the concept of negative of a point and also study some properties of even and odd convex functions. We apply Borsuk-Ulam theorem on $M$ and study the behavior of continuous function. The object of this paper is to prove the non-existence of odd convex functions on $M$ with a pole. We also obtain a characterization of even convex function on geodesics. In the last section we obtain some isoperimetric type inequality on $M$ with a pole whose sectional curvature is bounded by some constant.
\section{Notations and Symbols}
In this section we have discussed some rudimentary facts of $M$, which will be used in the sequel (for reference see \cite{JOS11}). The tangent space at the point $p\in M$ is denoted by $T_pM$ and $TM=\cup_{p\in M}T_pM$ is called the tangent bundle of $M$. Let $p, q\in M$ and $\gamma:[a,b]\rightarrow M$ be a smooth curve such that $\gamma(a)=p$ and $\gamma(b)=q$. The length $l(\gamma)$ of the curve $\gamma$ is given by
\begin{eqnarray*}
l(\gamma)&=&\int_{a}^{b}\sqrt{g_{\gamma(t)}(\dot{\gamma}(t),\dot{\gamma}(t))}\ dt\\
&=&\int_{a}^{b}\norm {\dot{\gamma}(t)}dt.
\end{eqnarray*}
The curve $\gamma$ is said to be a geodesic if $\nabla_{\dot{\gamma}(t)}\dot{\gamma}(t)=0\ \forall t\in [a,b]$, where $\nabla$ is the Riemannian connection of $g$.
For any point $p\in M$, the exponential map $exp_p:V_p\rightarrow M$ is defined by
$$exp_p(u)=\sigma_u(1),$$
where $\sigma_u$ is a geodesic with $\sigma(0)=p$ and $\dot{\sigma}_u(0)=u$ and $V_p$ is a collection of vectors of $T_pM$ such that for each element $u\in V_p$, the geodesic with initial tangent vector $u$ is defined on $[0,1]$. It is easy to see that for a geodesic $\sigma$, the norm of tangent vector $\norm {\dot{\gamma}(t)}$ is constant. A geodesic is said to be normal if its tangent vector is of unit norm. If for each point $p\in M$, the exponential map is defined at every point of $T_pM$, then $M$ is said to be a complete. Hopf-Rinow theorem states that there are some equivalent cases of completeness of $M$. Let $p$, $q\in M$. The distance between $p$ and $q$ is defined by
$$d(p,q)=\inf\{l(\gamma):\gamma \text{ be a curve joining }p \text{ and }q\}.$$
A geodesic $\sigma$ joining $p$ and $q$ is called minimal if $l(\sigma)=d(p,q)$. Hopf-Rinow theorem provides the existence of minimal geodesic between two points in a $M$. A smooth vector field is a smooth function $X:M\rightarrow TM$ such that $\pi\circ X=id_M$, where $\pi:TM\rightarrow M$ is the projection map. The integral curve of the vector field $X$ at the point $p$ is the solution $($always exists by Picard existence theorem$)$ of the differential equation
\[(1) \qquad\qquad\qquad\qquad \left\{
\label{eq}
\begin{array}{ll}
      \dot{x}(t)=X(x(t))   \\
      x(0)=p. \\
       \end{array} 
\right. \]
\begin{defn}\cite{UDR94}
 A real valued function $f$ on $M$ is called convex if for every geodesic $\gamma:[a,b]\rightarrow M$ the following inequality holds
\begin{equation*}
f\circ\gamma((1-t)a+tb)\leq (1-t)f\circ\gamma(a)+tf\circ\gamma(b)\quad \forall t\in [0,1].
\end{equation*}
\end{defn}

\section{Convex functions on the manifold with a pole}
A pole in $M$ is such a point where the tangent space is diffeomorphic to the whole manifold. Gromoll and Meyer \cite{GM69} introduced the notion of pole in $M$. A point $o\in M$ is called a pole \cite{GM69} of $M$ if the exponential map at $o$ is a global diffeomorphism and a manifold $M$ with a pole $o$ is denoted by $(M,o)$. If a manifold possesses a pole then the manifold is complete. Simply connected complete Riemannian manifold with non-positive sectional curvature and a paraboloid of revolution are the examples of Riemannian manifolds with a pole. A manifold with a pole is diffeomorphic to the Euclidean space but the converse is not true always, see \cite{ITO80}.\\
\indent For any point $x\in \mathbb{R}^n$, the negative of $x$ lies on the straight line passing through the origin and the point, and the distance between $-x$ and origin is same as that of $x$ and the origin. Similarly one can define the negative of a point in a manifold with a pole by considering the pole as the origin. Let $(M,o)$ be a manifold with a pole $o$. Then for each $r>0$ and $v\in\mathbb{S}^{n-1}$ the map $(r,v)\mapsto exp_o(rv)$ is a diffeomorphism from $(0,\infty)\times\mathbb{S}^{n-1}\rightarrow M-\{o\}$. Now we define the negative of a point as follows:
\begin{defn}
For a point $x\in (M,o)$ the \textit{negative} of $x$ with respect to $o$ is defined by
$$-_ox=exp_o(r_1v^a_1),$$
where $x=exp_o(r_1v_1)$ for some $r_1\in (0,\infty)$ and $v_1\in\mathbb{S}^{n-1}$ and $v^a_1$ is the antipodal point of $v_1$ in $\mathbb{S}^{n-1}$. 
\end{defn}

\begin{defn}
A function $f:(M,o)\rightarrow \mathbb{R}$ is said to be even with respect to $o$ if 
$$f(x)=f(-_ox),$$
and is called odd with respect to $o$ if 
$$f(x)=-f(-_ox),$$
for all $x\in M$.
\end{defn}
\begin{exmp}\label{ex1}
Let $o$ and $\bar{o}$ be two antipodal points in $\mathbb{S}^2$. Then $M=\mathbb{S}^2-\{\bar{o}\}$ is a Riemannian manifold with the canonical metric induced from $\mathbb{R}^3$ and $o$ is the pole. Now each $x\in M$ can be represented by $(x^i)_{i=1}^n$, where
$exp_o^{-1}(x)=x^i\partial_i$ 
and $\{\partial_1,\cdots,\partial_n\}$ is the standard basis for $T_oM$. If $f:(M,o)\rightarrow\mathbb{R}$ is defined by $f(x)=\sum_{i=1}^nx^i$ for $x\in (M,o)$, then it can be easily proved that $f$ is an odd function with respect to $o$.
\end{exmp}
\begin{exmp}
Consider the manifold $(M,o)$, defined in the Example \ref{ex1}. Then the distance function $d_o:(M,o)\rightarrow\mathbb{R}$ defined by
$$d_o(x)=L(\gamma_x),$$
where $L(\gamma_x)$ is the length of the minimal geodesic joining $o$ and $x$. Then $d_o$ is an example of even function on $(M,o)$. 
\end{exmp}
In this section an even or odd function means even or odd function with respect to a pole. \\
\indent If $M$ is simply connected with non-positive sectional curvature, then for any point $o\in M$, the square of distance function $d^2_o$, which is even function by our definition, is a convex function \cite{JOS11}. Kasue \cite{KAS81} obtained some conditions of the convexity of distance function on the manifold with a pole. Naturally the question arises that is there any odd convex function in a manifold with a pole? In this section we shall give a partial answer to this question. The main theorem of the paper is as follows:
\begin{thm}\label{thm2}
On a $(M,o)$ with dim$(M)>1$,there does not exist any non-negative differentiable odd convex function whose gradient is a complete vector field on $M$.
\end{thm}
The following results are needed to established this Theorem.

\begin{thm}[The Borsuk-Ulam Theorem]\cite{MAT03}
Let $f:\mathbb{S}^n\rightarrow\mathbb{R}^m$ be any continuous map, where $m\leq n$. Then there exist two antipodal points $x,\bar{x}\in\mathbb{S}^n$ such that $f(x)=f(\bar{x})$.
\end{thm}
 Two points $x$ and $\bar{x}$ in $\mathbb{S}^{n-1}$ are said to be antipodal if $x$ and $\bar{x}$ lie on the same diameter. 
\begin{lem}\label{lem1}
Let $f$ be a real valued continuous function on $M$ with dimension $>1$. Then each neighborhood $U_p$ of any point $p\in M$ contains uncountably many pairs of points where $f$ gives the same value.
\end{lem}
\begin{proof}
Let $V$ be a neighborhood $($ briefly, nbd $)$ of $p$ and $N_p$ be the normal nbd of $p$. Then there exists a diffeomorphism  $\varphi:(0,R)\times \mathbb{S}^{n-1}\rightarrow V_p-\{p\}$ defined by $\varphi(r,v)=exp_p(rv)$, where $V_p$ is an open nbd of $p$ containted in $V\cap N_p$ and $R\leq inj(p)$, $inj(p)$ is the injective radius of $p$. Hence the map $h:(0,R)\times\mathbb{S}^{n-1}\rightarrow\mathbb{R}$ defined by $h(r,v)=f(\varphi(r,v))$ is also continuous. Now for a fixed $r_0\in(0,R)$, the map $\varphi(r_0,.)$ is continuous and injective. Hence the function $h(r_0,.):\mathbb{S}^{n-1}\rightarrow\mathbb{R}$ is continuous. Now by Borsuk-Ulam theorem there exist two antipodal points $v$ and $\bar{v}$ in $\mathbb{S}^{n-1} $ such that $h(r_0,v)=h(r_0,\bar{v})$. Since $\varphi(r_0,.)$ is injective, hence $\xi=\varphi(r_0,v)$ and $\eta=\varphi(r_0,\bar{v}) $ are two distinct points in $V_p-\{p\} $ such that $f(\eta)=f(\xi)$. Now for $q\in V_p-\{p\}$ there exists $r_1\in (0,R)$ and $v_q\in\mathbb{S}^{n-1}$ such that $\varphi(r_1,v_q)=q$ and $r_1\neq r_0$. Hence by similar argument there exists $\bar{q}\in V_p-\{p\}$ such that $f(q)=f(\bar{q})$. Hence for each $r\in (0,R)$ there exists a pair in which $f$ gives the same value. Hence there are uncountably many such pairs in $U$.
\end{proof}
A simple observation from the characterization of odd function and Lemma \ref{lem1} is the following:
\begin{cor}\label{cor1}
If $f:(M,o)\rightarrow \mathbb{R}$ is an odd function, then $f$ vanishes at uncountably many antipodal points in $M$.
\end{cor}
\begin{thm}[Mean value theorem]\cite{AFL05}
Let $f$ be a real valued differential function on $M$. Then, for every pair of points $p,q\in M$ and every minimal geodesic path $\sigma:[0,1]\rightarrow M$ joining $p$ and $q$, there exists $t_0\in[0,1]$ such that
$$f(p)-f(q)=d(p,q)df_{\sigma(t_0)}(\dot{\sigma}(t_0)).$$
In particular, $|f(p)-f(q)|\leq \norm {df(\sigma(t_0))}_{\sigma(t_0)}d(p,q).$
\end{thm}
\begin{lem}\label{thm1}
If $f$ is a non-trivial real valued differentiable convex function on $M$, then there exists a geodesic $\sigma:[0,l]\rightarrow M$, for some $l>0$, and a point $t_2\in [0,l]$ such that 
\begin{equation}\label{eq4}
\langle grad(f)(\sigma(t_2)),\dot{\sigma}(t_2)\rangle\geq 1.
\end{equation}
\end{lem}
\begin{proof}
Let us consider the following negative gradient flow of the non-trivial convex function $f:M\rightarrow\mathbb{R}$:
\[   \left\{
\begin{array}{ll}
      \dot{x}(t)=-\text{grad }f(x(t))   \\
      x(0)=x. \\
       \end{array} 
\right. \]
The value of $f$ decreases along the orbit of $grad(f)$. Now
\begin{eqnarray*}
f(x)-f(x(t))&=&-\int_{0}^{t}\frac{d}{dt}f(x(t))dt\\
&=&-\int_{0}^{t}\norm {grad(f(x(t)))}^2dt\\
&=& \int_{0}^{t}\norm {\dot{x}(t)}^2dt.
\end{eqnarray*}
Let $\sigma:[0,t_0]\rightarrow M$ be the minimal geodesic connecting $x$ and $x(t_0)$ and parametrized in the interval $[0,t_0]$. Then 
\begin{eqnarray*}
d(x,x(t_0))&=&\int_{0}^{t_0}\norm {\dot{\sigma}(t)}dt\\
&=&\norm {\dot{\sigma}(0)}t_0\\
&\leq & \int_{0}^{t_0}\norm {\dot{x}(t)}dt\\
&\leq & \sqrt{t_0}\Big(\int_{0}^{t_0}\norm {\dot{x}(t)}^2 \Big)^{\frac{1}{2}}dt \quad [\text{ Using H\"{o}lder's inequality}]\\
&=& \sqrt{t_0}\Big (f(x)-f(x(t_0)) \Big)^{\frac{1}{2}}\\
&=& \sqrt{t_0}[d(x,x(t_0)) df(\sigma(t_1))(\dot{\sigma}(t_1))]^{\frac{1}{2}}\ \text{for some }0<t_1<t_0.
\end{eqnarray*}
The last inequality is obtained by using mean value theorem. From the above discussion we get
\begin{eqnarray*}
d(x,x(t_0))&\leq & t_0 df(\sigma(t_1))\dot{\sigma}(t_1)\\
\norm {\dot{\sigma}(0)}t_0&\leq & t_0\langle grad(f(\sigma(t_1)),\dot{\sigma}(t_1)\rangle.\quad [\text{ since $\sigma$ is the minimal geodesic,}]\\
\norm {\dot{\sigma}(0)}&\leq & \langle grad(f(\sigma(t_1)),\dot{\sigma}(t_1)\rangle.
\end{eqnarray*}
Since $\sigma$ is a geodesic, hence $\norm {\dot{\sigma}(0)}=\norm {\dot{\sigma}(t)}\ \forall t\in [0,t_0].$ So, we get

\begin{equation}\label{eq2}
\norm {\dot{\sigma}(t)}\leq \langle grad(f(\sigma(t_1)),\dot{\sigma}(t_1)\rangle \ \forall t\in [0,t_0].
\end{equation}
Due to the convexity of $f:M\rightarrow \mathbb{R}$ for $0\leq \lambda \leq 1$, we get
\begin{equation}
f\circ \sigma((1-\lambda)t_1+\lambda t_0)\leq (1-\lambda)f(\sigma(t_1))+\lambda f(\sigma(t_0)). 
\end{equation}
Now taking limit $\lambda\rightarrow 0$, we get $$\lim\limits_{\lambda \rightarrow 0}\frac{f\circ \sigma((1-\lambda)t_1+\lambda t_0)-f(\sigma(t_1))}{\lambda}\leq  f(\sigma(t_0))-f(\sigma(t_1)).$$ Hence,
\begin{equation}
\langle grad(f)(\sigma(t_1)),(\dot{\sigma}(t_1))\rangle\leq  \frac{f(\sigma(t_0))-f(\sigma(t_1))}{t_0-t_1}.
\end{equation}
So, from equation (\ref{eq2}) we say that
\begin{equation}
\norm {\dot{\sigma}(t)}\leq \frac{f(\sigma(t_0))-f(\sigma(t_1))}{t_0-t_1}.
\end{equation}
Now again applying mean value theorem to the right hand side of the above equation we get
\begin{eqnarray}
\norm {\dot{\sigma}(t)}&\leq & \frac{1}{t_0-t_1}d(\sigma(t_0),\sigma(t_1))df(\sigma(t_2))(\dot{\sigma}(t_2)),\quad \text{for some }t_1<t_2<t_0,\\
&\leq & \norm {\dot{\sigma}(t)}df(\sigma(t_2))(\dot{\sigma}(t_2)).
\end{eqnarray}
Hence 
\begin{equation}\label{eq3}
\langle grad(f)(\sigma(t_2)),\dot{\sigma}(t_2)\rangle\geq 1.
\end{equation}
\end{proof}
\begin{cor}
Let $f:M\rightarrow \mathbb{R}$ be a convex function such that $grad(f)$ is a complete vector field on $M$. Then equation (\ref{eq4}) holds for all geodesics on $M$.
\end{cor}
\begin{proof}
This can easily be proved by taking the minimal geodesic with initial point $x$ and final point $x(1)$.
\end{proof}
\begin{lem}\label{lem2}
Let $f:(M,o)\rightarrow\mathbb{R}$ $($ dim $M>1)$ be a non-negative odd convex function with $grad(f)$ as a complete vector field. Then for every geodesic $\sigma:[0,1]\rightarrow M$ with $\sigma(0)=o$, there exists a point $\sigma(t_0)$ for some $t_0\in [0,1]$ such that $f(\sigma(t_0))= 1$.
\end{lem}
\begin{proof}
Let $f:(M,o)\rightarrow\mathbb{R}$ be an odd convex function such that $grad(f)$ as a complete vector field on $M$. Then by virtue of Lemma \ref{thm1}, every geodesic $\sigma:[0,1]\rightarrow M$ has a  point $q=\sigma(t_0),$ where $0<t_0<1$, such that
\begin{equation}
\langle grad(f)(\sigma(t_0)),\dot{\sigma}(t_0)\rangle\geq 1.
\end{equation}
By convexity of $f$ we get
\begin{equation}
\langle grad(f)(\sigma(t_0)),\dot{\sigma}(t_0)\rangle\leq f(exp_{\sigma(t_0)}\dot{\sigma}(t_0))-f(\sigma(t_0)).
\end{equation}
If $f$ is non-negative, then from the above two equations we get
\begin{equation*}
f(exp_{\sigma(t_0)}\dot{\sigma}(t_0))\geq 1.
\end{equation*}
Since $f$ is an odd convex function, so $f(o)=0$. Hence by intermediate value theorem, on the geodesic $\sigma$ with initial point $o$, there is a point where $f$ takes the value $1$. The point $exp_{\sigma(t_0)}\dot{\sigma}(t_0)$ depends on the choice of geodesic $\sigma$. Hence for different geodesics we get different points. 
\end{proof}

\begin{proof}[Proof of Theorem \ref{thm2}:]
 Let $f:(M,o)\rightarrow\mathbb{R}$ be a non-negative differentiable odd convex function such that $grad(f)$ is complete. Then for any two points $p,q\in M$ and for the unique minimal geodesic $\sigma:[0,1]\rightarrow M$ with $\sigma(0)=p$ and $\sigma(1)=q$, we get
$$f(\sigma(t))\leq A\ \forall t\in [0,1], \text{ where }A=(1-t)f(p)+tf(q).$$
Now by taking negative in both sides of the above equation, we get
$$-f(\sigma(t))\geq -A\ \forall t\in [0,1],$$
which implies that
$$f(-\sigma(t))\geq (1-t)f(-p)+tf(-q)\ \forall t\in [0,1],$$
where $-\sigma:[0,1]\rightarrow M$ is the unique minimal geodesic such that $-\sigma(0)=-p$ and $-\sigma(1)=-q$. But since $f$ is convex, it follows that
$$f(\sigma(t))= A,$$
where $p,q\in M$ and $\sigma:[0,1]\rightarrow M$ is the minimal geodesic with initial point $p$ and final point $q$.
Again in view of Corollary \ref{cor1}, there exists a point $p\in M$ such that $f(p)=0$. So, by using convexity of $f$, we get
\begin{equation}
f(\sigma(t))=0\ \forall t\in [0,1],
\end{equation}
where $\sigma(t)$ is the minimal geodesic with $\sigma(0)=o$ and $\sigma(1)=p$, which contradicts to the Lemma \ref{lem2}. Hence there does not exist any non-negative differentiable odd convex function in $(M,o)$ whose gradient is a complete vector field.
\end{proof}

\begin{lem}\cite{PV17}
Let $f:\mathbb{R}^n\rightarrow\mathbb{R}$ be a continuous function whose lower level sets are bounded. Then $f$ has the global minimum.
\end{lem}
\begin{thm}
Let $f:\mathbb{R}^n(\simeq T_oM)\rightarrow\mathbb{R}$ be a continuous function whose lower level sets are
bounded. Then the function $h:(M,o)\rightarrow\mathbb{R}$ defined by
$$h(x)=f(exp_o^{-1}(x)), \text{ for }x\in M,$$
has the global minimum.
\end{thm}
\begin{proof}
Since $o\in M$ is a pole of $M$, hence the map $exp_o^{-1}:M\rightarrow T_oM$ is a global diffeomorphism. Now there is a natural isomorphism between $T_oM$ and $\mathbb{R}^n$ by using the standard orthonormal basis, so we can consider the function $f$ from $T_oM$ to $\mathbb{R}$. Since the lower level set of $f$ is bounded hence it has the global minimum $u_0\in T_oM$. We claim that $x_0=exp_o(u_0)$ is the global minimum point of $M$. \\
\indent Suppose there exists a point $y_0\in M$ such that $h(x_0)>h(y_0)$. But $h(y_0)=f(exp_o^{-1}(y_0))$. So by taking $v_0=exp_o^{-1}(y_0)$ we get that $f(u_0)>f(v_0)$, which is a contradiction since $u_0$ is the global minimum. Hence $x_0$ is the global minimum point of $h$.
\end{proof}

 \par Let $f:(M,o)\rightarrow\mathbb{R}$ be an even convex function. Then $f(x)=f(-_ox)\ \forall x\in M$. Let $\sigma:[0,1]\rightarrow M$ be a geodesic such that $\sigma(0)=x$ and $\sigma(1)=-_ox$. Then by convexity we get
\begin{equation*}
f(\sigma(t))\leq (1-t)f(x)+tf(-_ox)=f(x).
\end{equation*}
Hence
\begin{equation*}
\lim\limits_{t\rightarrow 0}\frac{f(\sigma(t))-f(x)}{t}=\langle grad(f)(\sigma(0)),\dot{\sigma}(0)\rangle \leq 0.
\end{equation*}
Consider the geodesic $\tilde{\sigma}:[0,1]\rightarrow M$ defined by $\tilde{\sigma}(t)=\sigma(1-t)$. Then $\dot{\tilde{\sigma}}(t)=-\dot{\sigma}(1-t)\ \forall t\in [0,1].$ Similarly, we get
$\langle grad(f)(\tilde{\sigma}(0)),\dot{\tilde{\sigma}}(0)\rangle \leq 0.$ This implies that 
$\langle grad(f)(\sigma(1)),\dot{\sigma}(1)\rangle \geq 0.$
Thus by intermediate value theorem there exists $t_0\in [0,1]$ such that $\langle grad(f)(\sigma(t_0)),\dot{\sigma}(t_0)\rangle = 0.$ Hence by changing the parameter of geodesic, we can conclude the following:
\begin{thm}
Let $f:(M,o)\rightarrow\mathbb{R}$ be an even convex function. Then for each normal geodesic $\sigma:[0,l]\rightarrow M$ connecting $x$ and $-_ox$, there exists a point $\sigma(t_0)$, for $0<t_0<l$, such that
\begin{equation}
\langle grad(f)(\sigma(t_0)),\dot{\sigma}(t_0)\rangle = 0.
\end{equation}
\end{thm}
\section{Isoperimetric type inequality in manifold with a pole}
 If $f:M\rightarrow \mathbb{R}$ is a convex function, then Greene and Wu \cite{GW71} proved that $f$ is subharmonic, i.e., $\Delta f\geq 0$. There is a mean value inequality on a Riemannian manifold \cite[Section 6]{SY94}. It compares the value of a
subharmonic function at the center of a ball with its mean value in the ball
\begin{thm}\cite{SY94}
Let $B_R(p)$ be a geodesic ball in $M$. Suppose that the sectional curvature $K_M\leq k$ for some constant $k$ and $R<inj(M,g)$. Then for any real valued smooth function $f$ with $\Delta f\geq 0$ and $f\geq 0$ on $M$,
\begin{equation}
f(p)\leq \frac{1}{V_k(R)}\int_{B_R(p)}f dV,
\end{equation}
where $V_k(R)=\text{Vol}(B_R,g_k)$ is the volume of a ball of radius $R$ in the space form of constant curvature $k$, $dV$ is the volume form and $inj(M,g)$ is the injective radius of $M$.
\end{thm}
We will use the above theorem in the manifold with a pole to prove an isoperimetric type inequality. Any ratio of volume and surface area of a compact subset is called isoperimetric type inequality, for more information see \cite{OSS78}. Since $(M,o)$ is a manifold with a pole, its injective radius is infinite. Here we will compare the volume of an unit ball in a manifold whose sectional curvature is a constant $k$ with surface area of unit ball in a manifold with sectional curvature less than $k$ with respect to a convex function. 
\begin{thm}
Let $(M,o)$ be a Riemannian manifold with a pole with sectional curvature $K_M\leq k$, for some constant $k$, and $f:(M,o)\rightarrow \mathbb{R}$ be a smooth convex function such that $f\geq 0$. Then
\begin{equation}
2\leq \frac{\omega_M}{V_k}+\frac{2}{V_k}\int_{\partial(B(o))}\frac{f(x)}{f(o)}dx,
\end{equation}
where $\omega_M$ is the surface area of unit ball in $M$ and $V_k$ is the volume of unit ball in the space form of constant curvature $k$ and $B(o)$ is the unit ball with center $o$.
\end{thm}
\begin{proof}
Since $f$ is convex, so $f$ is subharmonic function, i.e., $\Delta f\geq 0$. Now taking the unit ball $B(o)$ with center $o$ and applying the above theorem we get
\begin{eqnarray}
f(o)&\leq& \frac{1}{V_k}\int_{B(o)}f dV,\\
&=& \frac{1}{V_k}\int_{\partial B(o)}\Big\{\int_{0}^{1}f\circ \gamma_x(t)dt\Big\}dx,
\end{eqnarray}
where $\gamma_x:[0,1]\rightarrow M$ is the minimal geodesic such that $\gamma_x(0)=o$ and $\gamma_x(1)=x$ for each $x\in \partial B(o)$. By using convexity of $f$ we obtain
\begin{eqnarray*}
f(o)&\leq & \frac{1}{V_k}\int_{\partial B(o)}\Big\{\int_{0}^{1}((1-t)f(o)+tf(x))dt\Big\}dx\\
&=& \frac{1}{V_k}\int_{\partial B(o)}\Big\{f(o)\int_{0}^{1}(1-t)dt+f(x)\int_{0}^{1}tdt\Big\}dx\\
&=& \frac{1}{2V_k}\int_{\partial B(o)}(f(o)+f(x)) dx\\
&=& \frac{\omega_M}{2V_k}f(o)+\frac{1}{V_k}\int_{\partial(B(o))}f(x)dx.
\end{eqnarray*}
Since $f\geq 0$, hence dividing both sides by $f(o)$ we get our result.
\end{proof}
\section*{Acknowledgement}
 The second author greatly acknowledges to The UGC, Government of India for the award of JRF.

\end{document}